\documentclass[11pt]{article}
\usepackage{latexsym, amsmath, amssymb}

\usepackage{amsmath}
\usepackage{amsfonts}
\usepackage{amssymb}
\usepackage{amscd}
\usepackage{amsthm}
\usepackage{fancyhdr}
\theoremstyle{plain}
\newtheorem{theorem}{Theorem}[section]
\newtheorem{proposition}[theorem]{Proposition}
\newtheorem{remark}[theorem]{Remark}
\newtheorem{lemma}[theorem]{Lemma}
\newtheorem{definition}[theorem]{Definition}

\def\C{\mathbb C}

\def\Q{\mathbb Q}

\newcommand\sO{{\mathcal O}}

  \def\peen{\hbox{$ {\mathbf  P}^n$}}

  \def \tab#1{\kern #1 truein}

  \def\E{\hbox{${\cal E}$}}

  \def\A{\hbox{${{\cal A}}$}}
  \def\B{\hbox{${\cal B}$}}
  \def\C{\hbox{${\cal C}$}}

  \def\Q{\hbox{${\cal Q}$}}

  \def\O#1{\hbox{${\cal O}_{#1}$}}  

\begin{document}
 \title{Low Rank Vector Bundles on the Grassmannian $G(1,4)$} 
\author{Francesco Malaspina
\vspace{6pt}\\ 
 {\small  Politecnico di Torino}\\
{\small\it  Corso Duca degli Abruzzi 24, 10129 Torino, Italy}\\
{\small\it e-mail: malaspina@calvino.polito.it}}    \maketitle \def\thefootnote{}
\footnote{Mathematics Subject Classification 2000: 14F05, 14J60. \\ 
keywords: Universal Bundles on Grassmannians; Castelnuovo-Mumford regularity; Monads.}
  \begin{abstract}Here we define the concept of $L$-regularity for coherent sheaves on the Grassmannian $G(1,4)$ as a generalization of Castelnuovo-Mumford regularity on ${\bf{P}^n}$. In this setting we prove analogs of some classical properties. We use our notion of $L$-regularity in order to prove a splitting criterion for rank $2$ vector bundles with only a finite number of vanishing conditions. In the second part we give the classification of rank $2$ and rank $3$ vector bundles without "inner" cohomology (i.e. $H^i_*(E)=H^i(E\otimes\Q)=0$ for any $i=2,3,4$) on $G(1,4)$ by studying the associated monads.
  \end{abstract}  
  \section*{Introduction}
  In chapter $14$ of \cite{m} Mumford introduced the concept of regularity for a coherent sheaf on a projective space $\bf P^{n}$.
It was soon clear that Mumford's definition of Castelnuovo-Mumford regularity 
was a key notion and a fundamental tool in many areas of algebraic geometry and commutative algebra. It has shown a very powerful
tool, especially to study vector bundles. Chipalkatti generalized this notion to coherent sheaves on Grassmannians (\cite{c}) and  Costa and Mir\'o-Roig gave on any $n$-dimensional smooth projective varieties with an $n$-block collection (\cite{cm2}). In \cite{am}, it is introduced a simpler notion of regularity (called $G$-regularity) just on Grassmannians of lines by using the generalization of the Koszul exact sequence. It is a good tool because  it includes some vector bundles which are 
not regular in the sense of \cite{c} and can be use in order to characterize
direct sums of line bundles and give a
cohomological characterization of exterior and symmetric powers of
the universal bundles of the Grassmannian. Unfortunately this notion, consists of infinitely many cohomological vanishings.   However on $G(1,2)$ and $G(1,3)$ there are notions of regularity (which implies the $G$-regularity) with finite conditions: the Castelnuovo-Mumford regularity on 
 $G(1,2)\cong{\bf {P}}^2$ and the Qregularity on $G(1,3)\cong\Q_4$ (see \cite{bm}).\\
 In this paper we consider $G(1,4)$ and we give a notion of regularity with only a finite number of vanishing conditions. Next we show that the $L$-regularity implies the $G$-regularity and we prove the analogs of the classical properties on ${\bf{P}^{n+1}}$.\\ 
 A well-known result of Horrocks (see \cite{Ho}) characterizes the vector bundles without intermediate
cohomology on a projective space as direct sum of line bundles. This criterion fails on  more general varieties. There
exist non-split vector bundles without intermediate cohomology. These bundles are called ACM bundles. For instance the universal bundles of a Grassmannian are ACM. Ottaviani generalized Horrocks criterion to  quadrics
and Grassmannians by giving cohomological splitting conditions for vector bundles (see \cite{o1, o2}). Arrondo and Gra\~{n}a in \cite{ag} gave a cohomological  characterization of the universal bundles and a classification of ACM bundles on $G(1,4)$. In \cite{am} Arrondo and the author generalized   the first part of \cite{ag} by giving a cohomological characterization of exterior and symmetric powers of
the universal bundles on any grassmannian of lines.\\   
 Here we apply our notion of regularity in order to prove a splitting criterion for rank $2$ vector bundle (see Proposition \ref{p4}). We require the vanishing of the intermediate cohomology only for some particular twist. So  we have the analogous  of  \cite{cv} Corollary $1.8.$ on ${\bf {P}}^n$ and \cite{bm} Proposition $4.6.$ on $\Q_n$.\\
 
 In the second part of the paper we deal with monads. A monad on ${\bf{P}^n}$ or, more generally, on a projective variety $X$, is a complex of three vector bundles
$$0 \rightarrow A \xrightarrow{\alpha} B \xrightarrow{\beta} C \rightarrow 0$$
such that $\alpha$ is injective as a map of vector bundles and $\beta$ is surjective.
Monads have been studied by Horrocks, who proved (see \cite{Ho}) that every vector bundle on ${\bf{P}^n}$  is the homology of a suitable minimal monad.
This correspondence holds also on a projective variety $X$ ($\dim X\geq 3$) if we fix a very ample line bundle $\sO_X(1)$ (see \cite{ml1}). \\
Rao, Mohan Kumar and Peterson on ${\bf{P}^n}$ (see \cite{KPR1}), and the author on quadrics (see \cite{ml1, ml2}) gave a classification of rank $2$ and $3$  vector bundles  without inner cohomology (i.e. $H^1_*(E)= ...=H^{n-1}_*(E)=0$) by studying the associated minimal monads.\\
On $G(1,4)$ we say that a vector bundle is  without inner cohomology if $H^i_*(E)=H^i(E\otimes\Q)=0$ for any $i=2,3,4$. Then we classify the rank $2$ and $3$  vector bundles  without inner cohomology.
In particular we prove that there are no minimal monads with $A\not=0$ or $C\not=0$  associated to a rank $2$ and $3$  vector bundle  without inner cohomology.\\
 
We are grateful to E. Arrondo for the useful discussions and his comments.

 \section{Regularity on $G(1,4)$}
 Throughout the paper $\peen$ will denote the projective space
consisting of the one-dimensional quotients of the
$(n+1)$-dimensional vector space $V$ over an algebraically
closed field $\mathbb {K}$ with characteristic
zero. $G(1,4)$ (frequently
denoted just by $G$) will be the Grassmann variety of lines in ${\bf{P}^4}$. 
 We consider the universal exact sequence on $G=G(1,4)$:
     \begin{equation}\label{u}
     0\to S^\vee\to V\otimes{\cal O}_G\to Q\to 0 
     \end{equation}
     defining the universal bundles $S$ and $Q$ over $G$ of respective ranks
$3$ and $2$. We will also write
$\O{G}(1)=\bigwedge^2Q\cong\bigwedge^{3}S$. In particular, we have
natural isomorphisms
\begin{equation}\label{dual-Q}
S^j Q^\vee\cong(S^jQ)(-j)
\end{equation}
(where $S^j$ denotes the $j$-th symmetric power) and
\begin{equation}\label{dual-S}
\bigwedge^j S^{\vee}\cong\bigwedge^{3-j}S(-1).
\end{equation}
The second  exterior product in the left map
of (\ref{u}) is \begin{equation}\label{w2}
 0\to\bigwedge^{2} S^{\vee}\to \bigwedge^{2} V\otimes \O{G}\to V\otimes Q\to S^2 Q\to 0.\end{equation}
 Observe now that we can glue the dual of (\ref{u}) twisted by
$\O{G}(-1)$ with (\ref{w2}) and we obtain
\begin{equation}\label{+w2} 0\to \Q(-2)\to V^* \otimes \O{G}(-1)\to \bigwedge^{2} V\otimes \O{G}\to V\otimes Q\to      S^2 Q\to 0.\end{equation}
Let us consider also the dual sequence twisted by
$\O{G}(-3)$:
\begin{equation}\label{+w2d} 0\to S^2 Q(-3)\to V^*\otimes Q(-2)\to \bigwedge^{2} V^*\otimes \O{G}(-1)\to V\otimes \O{G}\to    Q\to 0.\end{equation}
If we glue (\ref{+w2}) with (\ref{u}) twisted by
$\O{G}(-2)$ we obtain
\begin{equation}\label{++w2}
0\to S^{\vee}(-2)\to  V\otimes \O{G}(-2)\to V^* \otimes \O{G}(-1)\to$$ $$\to \bigwedge^{2} V\otimes \O{G}\to V\otimes Q\to      S^2 Q\to 0.\end{equation}

We can also glue the dual of (\ref{w2}) twisted by
$\O{G}(-3)$ with (\ref{w2}) and we obtain \begin{equation}\label{w2+}
0\to\bigwedge^{2} S^{\vee}(-3)\to \bigwedge^{2} V\otimes \O{G}(-3)\to V\otimes Q(-3)\to$$ $$\to V^*\otimes Q(-2)\to \bigwedge^{2} V^*\otimes \O{G}(-1)\to V\otimes \O{G}\to    Q\to 0.\end{equation}
Let us consider also the dual sequence twisted by
$\O{G}(-4)$:
\begin{equation}\label{+d}
 0\to Q(-5)\to V^*\otimes \O{G}(-4)\to \bigwedge^{2} V\otimes \O{G}(-3)\to$$ $$\to V\otimes Q(-3)\to V^*\otimes Q(-2)\to \bigwedge^{2} V\otimes \O{G}(-1)\to    S^\vee\to 0.\end{equation}
Finally the top exterior product in the left map
of (\ref{u}) (twisted by  $\O{G}(-3)$) glued with the dual, it is the analogous  in $G$ of the long Koszul exact sequence in the
projective space. We have 
\begin{equation}\label{koz}0\to\sO_G(-4)\to\bigwedge^{3}
V\otimes{\O{G}}_G(-3)\to\bigwedge^{2}
V\otimes Q(-3)\to V\otimes S^{2}Q(-3)\to$$ $$\to V^*\otimes S^{2}Q(2)\to\bigwedge^{2}
V^*\otimes Q(-1)\to\bigwedge^{3}V^*\otimes\sO_G\to\sO_G(1)\to0.\end{equation}
\begin{remark}Let us notice that all the symmetric powers (except the last) that appear in sequence (\ref{++w2}) have order smaller than $2$. This is not true for the analog sequence  when $n>4$. For this reason the author in convinced that these ideas cannot be extended on $G(1,n)$ with $n>4$.\end{remark}

 We are ready to introduce our notion of regularity:\\

     \begin{definition}\label{d4} We say that a coherent sheaf $F$ on $G(1,4)$ is {\it $m$-$L$-regular} if the following conditions hold:
     \begin{enumerate}
     \item[i] $H^1(F(m-1))=H^2(F(m-2))=H^3(F(m-3))=H^4(F(m-3))=H^{5}(F(m-3))=H^5(F(m-4))=H^6(F(m-4))=0$.
     \item[ii] $H^{2}(F\otimes Q(m-2))=H^{3}(F\otimes Q(m-3))=H^{4}(F\otimes Q(m-3))=H^{4}(F\otimes Q(m-4))=H^{5}(F\otimes  Q(m-4))=0$.
     \item[iii] $H^{3}(F\otimes S^2 Q(m-3))=H^{4}(F\otimes S^2 Q(m-4))=H^{5}(F\otimes S^2 Q(m-5))=0$.
     
     \end{enumerate}
     We will say $L$-regular instead of $0$-$L$-regular.
     \end{definition}

  \begin{proposition}\label{p1} Let $F$ be a $L$-regular coherent sheaf on $G=G(1,4)$. For any $k\geq 0$,
  \begin{itemize}
  \item[(a)] $F(k)$ is $L$-regular.
  \item[(b)] $F(k)$ is generated by its global sections.
  \end{itemize}
  
  \end{proposition}
   \begin{proof}
  First of all let us prove that 
  $$H^{6}(F\otimes Q(-5))=H^{6}( F\otimes S^{2}Q(-6))=0$$
  From the sequence (\ref{+w2}),
   tensored  by $F(-3)$ we have that 
  $$H^{6}( F(-4))=H^{5}(F(-3))=H^{4}(F\otimes Q(-3))=H^{3}(F\otimes S^2 Q(-3))=0,$$ implies $H^{6}(F\otimes Q(-5))=0$.\\
  From (\ref{+w2d})
  tensored  by $F(-3)$ we have that 
  $$H^{6}( F\otimes Q(-5))=H^{5}(F(-4))=H^{4}(F(-3))=H^{3}(F\otimes Q(-3))=0,$$ implies $H^{6}(F\otimes S^2 Q(-6))=0$.\\

  Now let us show that
  $$H^1(F)=H^2(F(-1))=H^3(F(-2))=H^4(F(-2))=H^{5}(F(-2))=H^6(F(-3))=0.$$

  Let us consider the sequence (\ref{koz}) tensored  by $F(-1)$, since

$$H^{7}(F(-5))=H^{6}(
F(-4))=H^{5}(F\otimes Q(-4))=H^{4}( F\otimes S^{2}Q(-4))=$$ $$=H^{3}( F\otimes S^{2}Q(-3))=H^{2}(F\otimes Q(-2))=H^{1}(F(-1))=0,$$ we obtain $H^1(F)=0$.\\

If we tensor (\ref{koz}) by $F(-2)$, since

$$H^{7}(F(-6))=H^{6}(F\otimes Q(-5))=H^{5}( F\otimes S^{2}Q(-5))=$$ $$=H^{4}( F\otimes S^{2}Q(-4))=H^{3}(F\otimes Q(-3))=H^{2}(F(-2))=0,$$ we obtain $H^2(F(-1))=0$.\\

If we tensor (\ref{koz}) by $F(-3)$, since

$$H^{6}( F\otimes S^{2}Q(-6))=H^{5}( F\otimes S^{2}Q(-5))=H^{4}(F\otimes Q(-4))=H^{3}(F(-3))=0,$$ we obtain $H^3(F(-2))=0$.\\

Moreover, since

 $$H^{6}( F\otimes S^{2}Q(-5))=H^{5}(F\otimes Q(-4))=H^{4}(F(-3))=0,$$ we obtain $H^4(F(-2))=0$.\\

Since

 $$H^{6}(F\otimes Q(-4))=H^{5}(F(-3))=0,$$ we obtain $H^5(F(-2))=0$ and clearly $H^6(F(-3))=0$.\\

 Next we want show that $$H^{1}(F\otimes Q)=H^{2}(F\otimes Q(-1))=H^{3}(F\otimes Q(-2))=H^{4}(F\otimes (-2))=H^{5}(F\otimes  Q(-3))=0$$

 Let us consider the sequence (\ref{w2+}) tensored by $F(-3)$, since

 $$H^{6}(F(-4))=H^{5}(F(-3))=0,$$ we obtain $H^5(F\otimes Q(-3))=0$.\\
If we tensor (\ref{w2+}) by $F(-2)$, since

 $$H^{6}( F\otimes Q(-4))=H^{5}(F(-3))=H^{4}(F(-2))=0,$$ we obtain $H^4(F\otimes Q(-2))=0$.\\

Moreover, since

 $$H^{6}(F\otimes Q(-5))=H^{5}( F\otimes Q(-4))=H^{4}(F(-3))=H^{3}(F(-2))=0,$$ we obtain $H^3(F\otimes Q(-2))=0$.\\
 
If we tensor (\ref{w2+}) by $F(-1)$, since

 $$H^{6}( F(-4))=H^{5}(F\otimes Q(-4))=H^{4}(F\otimes Q(-3))=H^{3}(F(-2))=H^{1}(F(-1))=0,$$ we obtain $H^2(F\otimes Q(-1))=0$.\\

Let us prove finally that
$$H^{2}(F\otimes S^2 Q(-1))=H^{3}(F\otimes S^2 Q(-2))=H^{4}(F\otimes S^2 Q(-3))=H^{5}(F\otimes S^2 Q(-4))=0.$$ 

Let us consider the sequence (\ref{++w2}) tensored by $F(-4)$, since


 $$H^{6}(F(-4))=H^{5}(F\otimes Q(-4))=0,$$ we obtain $H^5(F\otimes S^2 Q(-4))=0$.\\
 
 Moreover, tensoring (\ref{++w2}) by $F(-3)$, since

 $$H^{6}(F(-4))=H^{5}(F(-3))=H^{4}(F\otimes Q(-3))=0,$$ we obtain $H^4(F\otimes S^2 Q(-3))=0$.\\
 
 If we tensor (\ref{++w2}) by $F(-2)$, since

 $$H^{6}(F(-4))=H^{5}(F(-3))=H^{4}(F(-2))=H^{3}(F\otimes Q(-2))=0,$$ we obtain $H^3(F\otimes S^2 Q(-2))=0$.\\

 $(b)$ We need the following lemma:
  \begin{lemma}\label{l1} Let $F$ be a $L$-regular coherent sheaf on $G$. Then, it is $G$-regular.
     \end{lemma}
     \begin{proof}We only need to show that, for any $k\geq 0$, $$H^{1}(F\otimes Q(k-1))=H^{2}(F\otimes S^2 Q(k-2))=0.$$
  From the sequence (\ref{w2})  tensored by $F(-4)$ we see that    $H^6(F\otimes\bigwedge^{2} S^{\vee}(-4))=0$. In fact $$H^6(F(-4))=H^5(F\otimes Q(-4))=H^4(F\otimes S^2 Q(-4))= 0.$$
 Let us tensorize the sequence (\ref{w2+}) by $F(-1)$. Since 
 
 $$H^6(F\otimes\bigwedge^{2} S^{\vee}(-4))=H^5(F(-4))=H^4(F\otimes Q(-4))=$$ $$=H^3(F\otimes Q(-3))=H^2(F(-2))=H^1(F(-1))=0,$$ we have $H^{1}(F\otimes Q(-1))=0$.\\

 From the sequence (\ref{u}) tensored by $F(-4)$ we see that    $H^6(F\otimes S^{\vee}(-4))=0$. In fact $$H^6(F(-4))=H^5(F\otimes Q(-4))= 0.$$
 Let us tensorize the sequence (\ref{++w2}) by $F(-2)$. Since 
 
 $$H^6(F\otimes S^{\vee}(-4))=H^5(F(-4))=H^4(F(-3))=H^3(F(-2))=H^2(F\otimes Q(-2))=0,$$ we have $H^{2}(F\otimes Q(-2))=0$.\\
 Now, since $F(k)$ is $L$-regular for any $k\geq 0$, we have the claimed vanishing for any $k\geq 0$.
  \end{proof}

  Since $F$ is $G$-regular then by \cite{am} Proposition $2.3.$ it is globally generated.
  \end{proof}
  
  \begin{definition}Let $F$ be a coherent sheaf on $G$. We define the $L$-regularity of $F$, $Lreg (F)$, as the least integer $m$ such that $F$ is $m$-$L$-regular. We set $Lreg (F)=-\infty$ if there is no such an integer.
  \end{definition}
      
      We can use the notion of $L$-regularity in order to prove a splitting criterion for rank $2$ vector bundles on $G$ with only a finite number of vanishing conditions:
      
      \begin{proposition}\label{p4}Let $E$ be a rank $2$ bundle on $G$ with $Lreg (E)= 0$. Let us assume that $$H^1(E(-2))=H^3(E(-4))=H^4(E(-4))=H^5(E(-5))=0,$$ and
     $$H^{2}(E\otimes Q(-3))=H^{3}(E\otimes Q(-4))=H^{4}(E\otimes Q(-5))=0.$$
Then $E\cong Q$ or $E\cong\sO\oplus\sO(a)$ with $a\geq 0$.
\end{proposition}
     \begin{proof}If we apply Le Potier vanishing theorem to a rank $2$ bundle on $G$ with $Lreg(E)=0$, we obtain
$H^i(E(k-3))=0$ for any $i\geq 2$ and any $k\geq 0$,
so we have $H^2(E(-3)=0$.\\
    Since $Lreg (E)= 0$, $E$ is $L$-regular but $E(-1)$ not. $E(-1)$ is not $L$-regular  if and only if one of the following conditions is satisfied:
     \begin{enumerate}
     \item[i] $H^{6}(E(-5))\not=0$,
     \item[ii] $H^{3}(E(-1)\otimes S^{2}Q(-3)) \not=0$,
     \item[iii]  $H^{5}(E(-1) \otimes Q(-4) ) \not=0$,
     \item[iv] $H^{4}(E(-1)\otimes S^{2}Q(-4)) \not=0$,
      \item[v] $H^{5}(E(-1)\otimes S^{2}Q(-5)) \not=0$.\end{enumerate}
     \bigskip
     Let us consider one by one the conditions:\\
     $(i)$ Let $H^{6}(E(-5))\not=0$, so $H^0(E^{\vee})\not=0$ and $\cal O$ is a direct summand of $E$. Then $E\cong\sO\oplus\sO(a)$ with $a\geq 0$.\\
     $(ii)$ 
     Let $H^{3}(E(-1)\otimes S^{2}Q(-3)) \not=0$. Let us consider the  exact sequence (\ref{+w2d}) tensored  by $E(-1)$. Since $$H^{3}( E\otimes Q(-3))=H^{2}(E(-2))=H^{1}(E(-1))=0,$$ we see that $H^{0}(E\otimes Q(-1))\not=0.$\\
    From the sequence (\ref{+w2}) tensored  by $E(-4)$ we have that 
  $$H^{6}( E(-5))=H^{5}(E(-4))=H^{4}(E\otimes Q(-4))=0,$$ implies $H^{6}(E\otimes Q(-6))\cong H^{3}(E(-1)\otimes S^{2}Q(-3))$. But $H^{6}(E\otimes Q(-6))\cong H^{0}(E^\vee\otimes Q)$.

   Let us consider the following commutative diagram of natural morphisms:
     
     $$\begin{array}{ccc} H^{3}(E\otimes S^{2}Q(-4)) \otimes H^{3}(E^{\vee}\otimes S^{2}Q(-3)) & \xrightarrow{\sigma}& H^{6}(S^{2}Q \otimes S^{2}Q(-7))\\
     \uparrow & &\uparrow \\
    H^0(E\otimes Q(-1)) \otimes H^{3}(  E^\vee\otimes S^{2}Q(-3))  & \xrightarrow{\mu}& H^{3}(Q\otimes S^{2}Q(-4)  )\cong\mathbb C\\
    \uparrow & &\uparrow \\
    H^0(E\otimes Q^\vee) \otimes H^{0}(E^{\vee}\otimes Q )& \xrightarrow{\tau}& H^{0}(Q\otimes Q^{\vee})\cong\mathbb C\\
    \uparrow{\cong}&&\uparrow{\cong}\\
    {Hom}(Q,E)\otimes{Hom}(E,Q)&\xrightarrow{\gamma}&{Hom}(Q,Q)\end{array}$$ The map $\sigma$ comes from Serre duality and it is not zero, the right vertical map are isomorphisms and the left vertical map are surjective so also the map $\tau$ is not zero. This map is naturally identified with the map $\gamma$ consisting just of the composition of homomorphisms. This means that the composition of the  following maps $$Q \rightarrow E \rightarrow Q$$ is not zero. Since the endomorphisms of $Q$ are multiplications by scalars, we can assume (after multiplying by a suitable scalar) that the above composition is the identity.
     Now  we can conclude that $E\cong Q$.\\
  Now we have to show that the conditions $(iii)$, $(iv)$ and $(v)$ are not possible.\\
     $(iii)$ Let  $H^{5}(E(-1) \otimes Q(-4) ) \not=0$. Since $$H^{6}( E(-5))=H^{5}(E(-5))=0$$ we have $$H^{5}(E(-1) \otimes Q(-4) )\cong H^{6}(E\otimes S^\vee(-5))\cong H^{0}(E^\vee\otimes S),$$ so $H^{0}(E^\vee\otimes S)\not=0.$\\
     On other hand let us tensorize the sequence (\ref{+d}) by $E$. Since 
 
 $$H^5(E(-4))=H^4(E(-3))=H^3(E\otimes Q(-3))=H^2(E\otimes Q(-2))=H^1(E(-1))=0,$$ we have $H^{0}(E\otimes S^\vee)=0$. So we can conclude that  $ S$ is a direct summand of $E$. But $S$ has rank $3$ then we have a contradiction. \\
 
 $(vi)$ First of all we claim that $H^{1}(E \otimes Q(-2) )=0$.\\
     If $H^{1}(E \otimes Q(-2) )\not=0$ in fact, by arguing as above, we can conclude that  $ S^{\vee} $ is a direct summand of $E(-1)$. But $S^\vee$ has rank $3$ then we have a contradiction. \\
     Let $H^{4}(E(-1)\otimes S^{2}Q(-4)) \not=0$. Let us consider the  exact sequence (\ref{+w2d}) tensored  by $E(-2)$. Since $$H^{4}( E\otimes Q(-4))=H^{3}( E\otimes Q(-4))=H^{2}(E(-3))=H^{1}(E(-2))=0,$$ we have that $$H^{4}(E(-1)\otimes S^{2}Q(-4))\cong H^{1}(E \otimes Q(-2) ).$$
     $(v)$ Let us consider the  exact sequence (\ref{+w2d}) tensored  by $E(-3)$. Since $$H^{5}( E\otimes Q(-5))=H^{4}(E(-4))=H^{3}(E(-3))=H^{2}( E\otimes Q(-3))=0,$$ we have that $H^{5}( E\otimes S^2 Q(-6))=0$.
    
\end{proof}
\begin{remark}
We found the analogous of \cite{cv} Corollary $1.8.$ and \cite{bm} Proposition $4.6.$ on $G$.
\end{remark}

\section{Rank $2$ and rank $3$ vector bundles without inner cohomology}
We introduce the following definition:
\begin{definition}We will call bundle without inner cohomology a bundle $E$ on $G$ with 
$$H^i_*(E)=H^i_*(E\otimes Q) =0, \textrm{ for any $i=2,3,4.$}$$

\end{definition}
In this section we classify all the rank $2$ and rank $3$ bundles without inner cohomology.\\
Now we introduce the following tool: the monads.\\
Let $\E$ be a vector bundle on  $G$. There is the corresponding minimal monad
  $$0 \rightarrow A \xrightarrow{\alpha} B
\xrightarrow{\beta} C \rightarrow 0,$$ where $A$ and $C$ are sums of line
bundles and $B$ satisfies:
\begin{enumerate}
\item $H^1_*(B)=H^{n-1}_*(B)=0$ 
\item $H^i_*(B)=H^i_*(E)$ \ ${}\forall i, 1<i<5
$.
\end{enumerate}
A monad will be called minimal if the maps $\alpha$
and $\beta$ are minimal: the surjective map $\beta$ is said
minimal if no direct summand of $\C$ is the image of a line
subbundle of $\B$.\\
An equivalent condition is that  no generator of $B$ can be
sent in a
generator of $C$.\\
$\alpha$ is minimal if the surjective $\alpha^{\vee}$ is minimal
as defined for $\beta$.\\
If $M$ is a finitely generated graded module over the homogeneous coordinate ring of $G$, $S_G$, we denote by $\beta_i(M)$ the  total Betti numbers of $M$. We will mainly use  $\beta_0(M)$ which give the number of minimal generators of $M$.\\
Recall that if $$M \to N \to 0$$ is a surjection of finitely generated graded $S_G$-modules, then $\beta_0(M) \geq \beta_0(N)$. Furthermore, if the inequality is strict, it means that a set of minimal generators of $M$ can be chosen in such a way that one of generators in the set maps to zero. \\

\begin{remark}\label{r2} By \cite{ml2} Theorem $2.2.$  any minimal monad $$ 0 \to \A \xrightarrow{\alpha} \B
\xrightarrow{\beta} \C \to 0,$$ such that $\A$ or $\C$ are not
zero, for a rank $r$ ($r\leq 3$) bundle with $H^2_*(E)=H^{4}_*(E)=0$, must satisfy the
following conditions:
\begin{enumerate}

\item $H^1_*(\wedge^2\B)\not=0$, $\beta_0(H^1_*(\wedge^2\B))\geq
\beta_0(H^0_*(S_2\C))$,  if $\C$ is not zero.
\vskip0.8truecm
\item $H^1_*(\wedge^2\B^{\vee})\not=0$, $ \beta_0(H^1_*(\wedge^2\B^{\vee}))\geq
\beta_0(H^0_*(S_2\A^{\vee}))$, if $\A$ is not zero.
\vskip0.8truecm
 \item $H^2_*(\wedge^2\B)=H^2_*(\wedge^2\B^{\vee}) =0$.
\end{enumerate}
  \end{remark}  
  \begin{remark}Here we list the only non-zero intermediate cohomology of the universal bundles when tensored with $Q$ 
  and $S^\vee$ (see \cite{ag} Table $1.3$):\\
  $$h^1(Q\otimes S^\vee)=h^5(S\otimes Q(-5))=h^2(S^\vee\otimes S^\vee)=1.$$ 
  \end{remark}
  We are ready to prove the main result of this section:\\
      \begin{theorem}\label{i1}On $G$ the only rank $r$ ($r\leq 3$) bundles without inner cohomology are (up to twist) the following:\begin{enumerate}
\item for $r=2$, $Q$ and the sums of line bundles,\\
\item for $r=3$, $Q\oplus\sO(a)$, $S$, $S^\vee$ and the sums of line bundles.\\
\end{enumerate}
\end{theorem}
\begin{proof}
First of all let us assume that $H^1_*(E)\not=0$ and $H^5_*(E)\not=0$. We can consider a  minimal monad for $E$,
$$ 0 \to A \xrightarrow{\alpha} B
\xrightarrow{\beta} C \to 0.$$ $B$  satisfies  all the hypothesis of \cite{ag} Theorem $2.4$ so it is a direct sum of bundles $S$, $S^\vee$, $Q$ and $\O{G}$ with some twist.\\
Moreover $B$ must satisfy the conditions $H^1_*(\wedge^2 B)\not=0$ and $H^1_*(\wedge^2 B^\vee)\not=0$.  Since  $\wedge^2 S^\vee$, $\wedge^2 S$ and $\wedge^2 Q$ are all ACM bundles and the only non-zero $H^1$ cohomology of the tensor product between  universal bundles is
  $h^1(Q\otimes S^\vee)=1$, $B$ must have at least a copy of $Q$, $S$ and $S^\vee$.\\
 Assume that  more than one copy of $S^\vee$ or more than one copy of $S$  appears in $B$. Then in the
bundle $\wedge^2 B$ or in the bundle $\wedge^2 B^\vee$, it appears
$(S^\vee\otimes S^\vee)(t)$  and, since
$$H^2_*(S^\vee\otimes S^\vee)\not=0,$$  the condition
$$H^2_*(\wedge^2 B)=H^2_*(\wedge^2 B^\vee)=0$$ in Remark \ref{r2}, fails to be satisfied.\\  We can conclude that $B$  has to be of the form $$ 
 (\bigoplus_{i=1}^h\sO(a_i))\oplus(\bigoplus_{j=1}^k Q(b_j))\oplus ( S(c))\oplus (S^\vee(d)),$$ with $h\geq 0$ and $k\geq 1$.\\

  Let us notice
furthermore that $rank(B)=h+2k+6$ and $H^1_*(\wedge^2 B)\cong H^1_*((\bigoplus_{j=1}^k Q)\otimes S^\vee)$ has  $k$ generators. Since $rank(C)= h+2k+6-rank(E)-rank(A)$, we have $$\beta_0(H^0_*(S_2 C))\geq \beta_0(H^0_*(C))=h+2k+6-rank(E)-rank(A)\geq h+2k +3-rank(A).$$
So $k=\beta_0(H^1_*(\wedge^2 B))\geq
\beta_0(H^0_*(S_2 C))\geq h+2k+4+rank(A)$ 
which implies $rank(A)\geq h+k+3$.\\
Moreover $H^1_*(\wedge^2 B^\vee)\cong H^1_*((\bigoplus_{j=1}^k Q)\otimes S)$ has  $k$ generators.
So $k=\beta_0(H^1_*(\wedge^2 B))\geq
\beta_0(H^0_*(S_2 A))\geq rank(A)\geq h+k+3$ 
which is impossible.\\

Let us assume now that $H^1_*(E)\not=0$ and $H^5_*(E)=0$ (hence $rank(E)=3$). By using the above argument we see that, since $H^1_*(\wedge^2 B)\not=0$, at least one copy of $S^\vee$ must appear in $B$.
Moreover, since $H^2_*(\wedge^2 B)=0$, it is no possible to have more than one copy of $S^\vee$
We can conclude that $B$  has to be of the form $$ 
 (\bigoplus_{i=1}^h\sO(a_i))\oplus(\bigoplus_{j=1}^k Q(b_j))\oplus ((\bigoplus_{l=1}^s S(c_l))\oplus (S^\vee(d)),$$ with $h,s\geq 0$ and $k\geq 1$.\\

  Let us notice
furthermore that $rank(B)=h+2k+3s+3$ and $H^1_*(\wedge^2 B)\cong H^1_*((\bigoplus_{j=1}^k Q)\otimes S^\vee)$ has  $k$ generators. Since $rank(C)= h+2k+3s$, we have $$\beta_0(H^0_*(S_2 C))\geq \beta_0(H^0_*(C))=h+2k+3s.$$
So $k=\beta_0(H^1_*(\wedge^2 B))\geq
\beta_0(H^0_*(S_2 C))\geq h+2k+3s,$ 
which it is impossible.\\
A symmetric argument show that there are no minimal monads in the case $H^1_*(E)=0$ and $H^5_*(E)\not=0$.\\
 We proved that the every rank $r$ ($r\leq 3$) bundle without inner cohomology must have $H^1_*(E)=H^5_*(E)=0$. Then by \cite{ag} Theorem $2.4$ they are the claimed.
   \end{proof}


\begin{thebibliography}{99}
  
  \bibitem{ag} {\sc E. Arrondo and B. Gra\~{n}a},
\emph{Vector bundles on $G(1,4)$ without intermediate cohomology}, 1999, J. of Algebra 214, 128-142.
 \bibitem{am} {\sc E. Arrondo and F. Malaspina},
\emph{Cohomological Characterization of Vector Bundles on Grassmannians of Lines}, 2009, preprint arXiv:0902.2897. 

  
\bibitem{bm} {\sc E. Ballico and F. Malaspina,} \emph{ Qregularity and an Extension of the Evans-Griffiths Criterion to Vector Bundles on Quadrics}, J. Pure Appl Algebra 213 (2009), 194-202.
\bibitem{cv} {\sc L. Chiantini, P. Valabrega},
\emph{Subcanonical curves and complete intesections in the projective 3-space}, 1984, Ann. Mat. Pura Appl. 138, 309-330.

\bibitem{c} {\sc J. V. Chipalkatti,} \emph{A generalization of Castelnuovo 
regularity to Grassman varieties,} Manuscripta Math.
102 (2000), no. 4, 447-464.


\bibitem{cm2}{\sc L. Costa and R. M. Mir\'{o}-Roig,} \emph{$m$-blocks collections
and Castelnuovo-Mumford regularity
in multiprojective spaces,} Nagoya Math. J. 186 (2007), 119--155.

\bibitem
{Ho} {\sc G. Horrocks,}
\emph{Vector bundles on the punctured spectrum of a ring,} Proc. London Math. Soc. (3) 14 (1964), 689-713.

\bibitem
{KPR1}{\sc N.Mohan Kumar, C. Peterson and A.P. Rao},
\emph{Monads on projective spaces}, 2003, Manuscripta Math. 112, 183-189.

\bibitem{ml1}{\sc F. Malaspina},
\emph{Monads and Vector Bundles on Quadrics}, Adv. Geom. Vol. 9, issue 1, 2009, 137-152.

\bibitem{ml2} {\sc F. Malaspina},
\emph{Monads and rank three vector bundles on quadrics}, Ann. Mat. Pura  Appl. (2009) 188: 455-465.

\bibitem{m} {\sc D. Mumford,} \emph{Lectures on curves on an algebraic surface, 
Princeton University Press,} Princeton, N.J., 1966.

\bibitem{o1}{\sc G. Ottaviani}, 
\emph{Criteres de scindage pour les fibres vectoriel sur les grassmanniennes et les quadriques }, 1987, C. R. Acad. Sci. Paris, 305, 257-260.

\bibitem
{o2}{\sc G. Ottaviani},
\emph{Some extension of Horrocks criterion to vector bundles on Grassmannians and quadrics}, 1989, Annali Mat. Pura Appl. (IV) 155, 317-341.

\end{thebibliography}
\end{document}